\documentclass[11pt,oneside]{amsart}
\usepackage{amsthm}
\usepackage{amssymb}

\makeatletter
\numberwithin{equation}{section}
\numberwithin{figure}{section}
\theoremstyle{plain}
\newtheorem{thm}{Theorem}[section]
  \theoremstyle{definition}
  \newtheorem{defn}[thm]{Definition}
  \theoremstyle{plain}
  \newtheorem{cor}[thm]{Corollary}
  \theoremstyle{plain}
  \newtheorem{prop}[thm]{Proposition}
  \theoremstyle{plain}
  \newtheorem{lem}[thm]{Lemma}
  \theoremstyle{remark}
  \newtheorem{rem}[thm]{Remark}

\usepackage[unicode=true, pdfusetitle,
 bookmarks=true,bookmarksnumbered=false,bookmarksopen=false,
 breaklinks=false,pdfborder={0 0 1},backref=false,colorlinks=false]{hyperref}

\makeatletter

\newcommand{\cAq}{ {\mathcal B}^{(q)} } 
\newcommand{\cB}{ {\mathcal B} } 
\newcommand{\cBq}{ {\mathcal B}^{(q)} } 
\newcommand{\cC}{ {\mathcal C} } 
\newcommand{\cCq}{ {\mathcal C}^{(q)} } 
\newcommand{\cE}{ {\mathcal E} } 
\newcommand{\cEq}{ {\mathcal E}^{(q)} } 
\newcommand{\cF}{ {\mathcal F} } 
\newcommand{\cFq}{ {\mathcal F}^{(q)} } 
\newcommand{\cH}{ {\mathcal H} } 
\newcommand{\cS}{ {\mathcal S} } 
\newcommand{\cSq}{ {\mathcal S}^{(q)} }

\newcommand{\Aq}{ A^{(q)} } 
\newcommand{\Jq}{ J^{(q)} } 
\newcommand{\Lq}{ L^{(q)} } 
\newcommand{\Mq}{ M^{(q)} } 
\newcommand{\Mqplus}{ M_{+}^{(q)} } 
\newcommand{\Pq}{ P^{(q)} } 
\newcommand{\Rq}{ R^{(q)} }

\newcommand{\Uopp}{ U_{op} } 
\newcommand{\bC}{ {\mathbb C} } 
\newcommand{\inv}{ {\mathrm{inv}} }

\makeatother

\begin{document}
\global\long\def\cAq{\mathcal{A}^{(q)}}

\global\long\def\cB{\mathcal{B}}

\global\long\def\cBq{\mathcal{B}^{(q)}}

\global\long\def\cC{\mathcal{C}}

\global\long\def\cCq{\mathcal{C}^{(q)}}

\global\long\def\cE{\mathcal{E}}

\global\long\def\cEq{\mathcal{E}^{(q)}}

\global\long\def\cF{\mathcal{F}}

\global\long\def\cFq{\mathcal{F}^{(q)}}

\global\long\def\cH{\mathcal{H}}

\global\long\def\cS{\mathcal{S}}

\global\long\def\cSq{\mathcal{S}^{(q)}}

\global\long\def\Aq{A^{(q)}}

\global\long\def\Jq{J^{(q)}}

\global\long\def\Lq{L^{(q)}}

\global\long\def\Mq{M^{(q)}}

\global\long\def\Mqplus{\widehat{M}^{(q)}}

\global\long\def\Pq{P^{(q)}}

\global\long\def\Rq{R^{(q)}}

\global\long\def\Uopp{U_{opp}}

\global\long\def\bC{\mathbb{C}}

\global\long\def\inv{\mathrm{inv}}

\global\long\def\cM{\mathcal{M}}

\global\long\def\cMq{\mathcal{M}^{(q)}}

\title[Exactness of the $q$-commutation relations]{Exactness of the Fock space representation \\
 of the $q$-commutation relations}

\author{Matthew Kennedy}

\address{Matthew Kennedy \\
 Pure Mathematics Department \\
 University of Waterloo \\
 Waterloo, Ontario, Canada N2L 3G1}

\email{m3kennedy@uwaterloo.ca}

\author{Alexandru Nica}

\address{Alexandru Nica \\
 Pure Mathematics Department \\
 University of Waterloo \\
 Waterloo, Ontario, Canada N2L 3G1}

\email{anica@math.uwaterloo.ca}

\begin{abstract}
We show that for all $q$ in the interval $(-1,1)$, the Fock representation
of the $q$-commutation relations can be unitarily embedded into the
Fock representation of the extended Cuntz algebra. In particular,
this implies that the $\mathrm{C}^{*}$-algebra generated by the Fock
representation of the $q$-commutation relations is exact. An immediate
consequence is that the $q$-Gaussian von Neumann algebra is weakly
exact for all $q$ in the interval $(-1,1)$. 
\end{abstract}

\thanks{Research supported by a CGS Scholarship from NSERC Canada (M. Kennedy) and by a Discovery Grant from NSERC Canada (A. Nica)}

\subjclass[2000]{Primary 46L05; Secondary 46L10, 46L54}

\maketitle

\section{Introduction}

The $q$-commutation relations provide a $q$-analogue of the bosonic
($q=1$) and the fermionic ($q=-1$) commutation relations from quantum
mechanics. These relations have a natural representation on a deformed
Fock space which was introduced by Bozejko and Speicher in \cite{BS1990},
and was subsequently studied by a number of authors (see e.g. \cite{BKS1997},
\cite{DN1993}, \cite{JSW1993}, \cite{N2004}, \cite{R2005}, \cite{S2004}).

For the entirety of this paper, we fix an integer $d\geq2$. Consider
the usual full Fock space $\cF$ over $\bC^{d}$, \begin{equation}
\mathcal{F}=\oplus_{n=0}^{\infty}\mathcal{F}_{n}\quad\mbox{ (orthogonal direct sum),}\label{eqn:non-deformed-fock}\end{equation}
 where $\mathcal{F}_{0}=\bC\Omega$ and $\mathcal{F}_{n}=(\mathbb{C}^{d})^{\otimes n}$
for $n\geq1$.

Corresponding to the vectors in the standard orthonormal basis of
$\mathbb{C}^{d}$, one has left creation operators $L_{1},...,L_{d}\in B(\cF)$.
Define the $\mathrm{C}^{*}$-algebra $\cC$ by \begin{equation}
\cC:=C^{*}(L_{1},\ldots,L_{d})\subseteq B(\cF).\label{eqn:non-deformed-cuntz}\end{equation}
 It is well known that $\cC$ is isomorphic to the extended Cuntz
algebra. (Although it is customary to denote the extended Cuntz algebra
by $\cE$, we use $\cC$ here to emphasize that we are working with
a concrete $C^{*}$-algebra of operators.)

Now let $q\in(-1,1)$ be a deformation parameter. We consider the
$q$-deformation $\cFq$ of $\cF$ as defined in \cite{BS1990}. Thus
\begin{equation}
\cFq=\oplus_{n=0}^{\infty}\cFq_{n}\quad\mbox{ (orthogonal direct sum),}\label{eqn:deformed-fock}\end{equation}
 where every $\cFq_{n}$ is obtained by placing a certain deformed
inner product on $(\mathbb{C}^{d})^{\otimes n}$. (The precise definition
will be reviewed in Subsection \ref{sub:q-fock-framework} below.)
For $q=0$, one obtains the usual non-deformed Fock space $\cF$ from
above.

In this deformed setting, one also has natural left creation operators
$\Lq_{1},...,\Lq_{d}\in B(\cFq)$, which satisfy the $q$-commutation
relations\[
\Lq_{i}(\Lq_{j})^{*}=\delta_{ij}I+q(\Lq_{j})^{*}\Lq_{i},\quad1\leq i,j\leq d.\]
 Define the $\mathrm{C}^{*}$-algebra $\cCq$ by \begin{equation}
\cCq:=C^{*}(\Lq_{1},\ldots,\Lq_{d})\subseteq B(\cFq).\label{eqn:deformed-cuntz}\end{equation}
 For $q=0$, this construction yields the extended Cuntz algebra $\cC$
from above.

It is widely believed that the algebra $\cC$ and the deformed algebra
$\cCq$ are actually unitarily equivalent. In fact, this is known
for sufficiently small $q$. In \cite{DN1993}, a unitary $U:\cFq\to\cF$
was constructed which embeds $\cC$ into $\cCq$ for all $q\in(-1,1)$,
i.e. $\cC\subseteq U\cCq U^{*}$, and it was shown that for $|q|<0.44$
this embedding is actually surjective, i.e. $\cC=U\cCq U^{*}$.

The main purpose of the present paper is to show that it is possible
to unitarily embed $\cCq$ into $\cC$ for all $q\in(-1,1)$. Specifically,
we construct a unitary operator $\Uopp:\cFq\to\cF$ such that $\Uopp\cCq\Uopp^{*}\subseteq\cC.$
The unitary $\Uopp$ is closely related to the unitary $U$ from \cite{DN1993},
as we will now see. 
\begin{defn}
\label{def:conjugation-operator} Let $J:\cF\to\cF$ be the unitary
conjugation operator which reverses the order of the components in
a tensor in $(\bC^{d})^{\otimes n}$, i.e. \begin{equation}
J(\eta_{1}\otimes\cdots\otimes\eta_{n})=\eta_{n}\otimes\cdots\otimes\eta_{1},\quad\forall\eta_{1},\ldots,\eta_{n}\in\bC^{d}.\label{eqn:1.11}\end{equation}
 Note that for $n=0$, Equation (\ref{eqn:1.11}) says that $J(\Omega)=\Omega$.

Let $\Jq:\cFq\to\cFq$ be the operator which acts as in Equation (\ref{eqn:1.11}),
where the tensor is now viewed as an element of the space $\cFq_{n}$.
It is known that $\Jq$ is also unitary operator (see the review in
Subsection \ref{sub:q-fock-framework}).
\end{defn}

\begin{defn}
\label{def:Uopp}Let $q\in(-1,1)$ be a deformation parameter and
let $U:\cFq\to\cF$ be the unitary defined in \cite{DN1993}. Define
a new unitary $\Uopp:\cFq\to\cF$ by \[
\Uopp=JU\Jq.\]

\end{defn}
The following theorem is the main result of this paper. 
\begin{thm}
\label{thm:main-inclusion} For every $q\in(-1,1)$ the unitary $\Uopp$
from Definition \ref{def:Uopp} satisfies\[
\Uopp\cCq\Uopp^{*}\subseteq\cC.\]
 
\end{thm}
The following corollary follows immediately from Theorem \ref{thm:main-inclusion}. 
\begin{cor}
\label{cor:Cq-exact} For every $q\in(-1,1)$ the $C^{*}$-algebra
$\cCq$ is exact. 
\end{cor}
To prove Theorem \ref{thm:main-inclusion}, we first consider the
more general question of how to verify that an operator $T\in B(\cF)$
belongs to the algebra $\cC$. It is well known that a necessary condition
for $T$ to be in $\cC$ is that it commutes modulo the compact operators
with the $\mathrm{C}^{*}$-algebra generated by right creation
operators on $\cF$. Unfortunately, this condition isn't sufficient
(and wouldn't be sufficient even if we were to set $d$ equal to $1$,
cf. \cite{D1977}). Nonetheless, by restricting our attention to a
$*$-subalgebra of ``band-limited operators'' on $\cF$ and considering
commutators modulo a suitable ideal of compact operators in this algebra,
we do obtain a sufficient condition for $T$ to belong to $\cC$.
This bicommutant-type result is strong enough to help in the proof
of Theorem \ref{thm:main-inclusion}.

In addition to this introduction, the paper has four other sections.
In Section 2, we provide a brief review of the requisite background
material. In Section 3, we prove the above-mentioned bicommutant-type
result, Theorem \ref{thm:inclusion-criterion}. In Section 4, we establish
the main results, Theorem \ref{thm:main-inclusion} and Corollary
\ref{cor:Cq-exact}. In Section 5, we apply these results to the family
of $q$-Gaussian von Neumann algebras, showing in Theorem \ref{thm:weakly-exact}
that these algebras are weakly exact for every $q\in(-1,1)$.

\section{Review of background}

\subsection{\label{sub:q-fock-framework}Basic facts about the $q$-deformed
Fock space}

As explained in the introduction, there is a fairly large body of research
devoted to the $q$-deformed Fock framework and its generalizations.
Here we provide only a brief review of the terminology and facts which
will be needed in Section \ref{sec:main-result}.

\subsubsection{The $q$-deformed inner product}

As mentioned above, the integer $d\geq2$ will remain fixed throughout
this paper. Also fixed throughout this paper will be an orthonormal
basis $\xi_{1},\ldots,\xi_{d}$ for $\bC^{d}$. For every $n\geq1$
this gives us a preferred basis for $(\bC^{d})^{\otimes n}$, namely
\begin{equation}
\{\xi_{i_{1}}\otimes\cdots\otimes\xi_{i_{n}}\mid1\leq i_{1},\ldots,i_{n}\leq d\}.\label{eqn:natural-basis}\end{equation}
 This basis is orthonormal with respect to the usual inner product
on $(\bC^{d})^{\otimes n}$ (obtained by tensoring $n$ copies of
the standard inner product on $\bC^{d}$). As in the introduction,
we will use $\cF_{n}$ to denote the Hilbert space $(\bC^{d})^{\otimes n}$
endowed with this inner product. The full Fock space over $\bC^{d}$
is then the Hilbert space $\cF$ from Equation (\ref{eqn:non-deformed-fock}),
with the convention that $\cF_{0}=\bC\Omega$ for a distinguished
unit vector $\Omega$, referred to as the {}``vacuum vector''.

Now let $q\in(-1,1)$ be a deformation parameter. It was shown in
\cite{BS1990} that there exists a positive definite inner product
$\langle\cdot,\cdot\rangle_{q}$ on $(\bC^{d})^{\otimes n}$, uniquely
determined by the requirement that for vectors in the natural basis
(\ref{eqn:natural-basis}), one has the formula \begin{equation}
\langle\xi_{i_{1}}\otimes\cdots\otimes\xi_{i_{n}},\xi_{j_{1}}\otimes\cdots\otimes\xi_{j_{n}}\rangle_{q}=\sum_{\sigma}q^{\inv(\sigma)}\delta_{i_{1},\sigma(j_{1})}\cdots\delta_{i_{n},\sigma(j_{n})}.\label{eqn:def-inner-prod}\end{equation}
 The sum on the right-hand side of Equation (\ref{eqn:def-inner-prod})
is taken over all permutations $\sigma$ of $\{1,\ldots,n\}$, and
$\inv(\sigma)$ denotes the number of inversions of $\sigma$, i.e.
\[
\inv(\sigma):=\left|\{(i,j)\mid1\leq i<j\leq n,\ \sigma(i)>\sigma(j)\}\right|.\]
 Note that under this new inner product, the natural basis (\ref{eqn:natural-basis})
will typically no longer be orthogonal.

We will use $\cFq_{n}$ to denote the Hilbert space $(\bC^{d})^{\otimes n}$
endowed with this deformed inner product. In addition, we will use
the convention that $\cFq_{0}$ is the same as $\cF_{0}$, i.e. it
is spanned by the same vacuum vector $\Omega$. The $q$-deformed
Fock space over $\bC^{d}$ is then the Hilbert space $\cFq$ from
Equation (\ref{eqn:deformed-fock}). For $q=0$, the construction
of $\cFq$ yields the usual non-deformed Fock space $\cF$ from Equation
(\ref{eqn:non-deformed-fock}).

\subsubsection{\label{sub:deformed-creation-annih-ops}The deformed creation and
annihilation operators}

For every $1\leq j\leq d$, one has deformed left creation operators
$L_{j}^{(q)}\in\cB(\cFq)$ and deformed right creation operators $R_{j}^{(q)}\in\cB(\cFq)$,
which act on the natural basis of $\cFq_{n}$ by $\Lq_{j}(\Omega)=\Rq_{j}(\Omega)=\xi_{j}$
and \begin{equation}
\left\{ \begin{array}{l}
\Lq_{j}(\xi_{i_{1}}\otimes\cdots\otimes\xi_{i_{n}})=\xi_{j}\otimes\xi_{i_{1}}\otimes\cdots\otimes\xi_{i_{n}},\\
\Rq_{j}(\xi_{i_{1}}\otimes\cdots\otimes\xi_{i_{n}})=\xi_{i_{1}}\otimes\cdots\otimes\xi_{i_{n}}\otimes\xi_{j}.\end{array}\right.\label{eqn:creation}\end{equation}
 Their adjoints are the deformed left annihilation operators $(\Lq_{j})^{*}$
and the deformed right annihilation operators $(\Rq_{j})^{*}$, which
act on the natural basis of $\cFq_{n}$ by\begin{equation}
\left\{ \begin{aligned}(\Lq_{j})^{*} & (\xi_{i_{1}}\otimes\cdots\otimes\xi_{i_{n}})\\
 & =\sum_{m=1}^{n}q^{m-1}\delta_{j,i_{m}}\xi_{i_{1}}\otimes\cdots\otimes\widehat{\xi_{i_{m}}}\otimes\cdots\otimes\xi_{i_{n}},\\
(\Rq_{j})^{*} & (\xi_{i_{1}}\otimes\cdots\otimes\xi_{i_{n}})\\
 & =\sum_{m=1}^{n}q^{n-m}\delta_{i_{m},j}\xi_{i_{1}}\otimes\cdots\otimes\widehat{\xi_{i_{m}}}\otimes\cdots\otimes\xi_{i_{n}},\end{aligned}
\right.\label{eqn:annihilation}\end{equation}
 where the {}``hat'' symbol over the component $\xi_{i_{m}}$ means
that it is deleted from the tensor (e.g. $\xi_{i_{1}}\otimes\widehat{\xi_{i_{2}}}\otimes\xi_{i_{3}}$
= $\xi_{i_{1}}\otimes\xi_{i_{3}}$).

It's clear from these formulas that the left creation (left annihilation)
operators commute with the right creation (right annihilation) operators.
For the commutator of a left annihilation operator and a right creation
operator, a direct calculation (see also Lemma 3.1 from \cite{S2004})
gives the formula \begin{equation}
[(\Lq_{i})^{*},\Rq_{j}]\mid{}_{\cFq_{n}}=\delta_{ij}q^{n}I_{\cFq_{n}},\quad\forall n\geq1.\label{eqn:comm-relns}\end{equation}
 Taking adjoints gives the formula for the commutator of a left creation
operator and a right annihilation operator.

When we are working on the non-deformed Fock space $\cF$ corresponding
to the case when $q=0$, it will be convenient to suppress the superscripts
and write $L_{j}$ and $R_{j}$ for the left and right creation operators
respectively. Note that in this case, Equation (\ref{eqn:creation})
and Equation (\ref{eqn:annihilation}) imply that \begin{equation}
\sum_{j=1}^{d}L_{j}L_{j}^{*}=\sum_{j=1}^{d}R_{j}R_{j}^{*}=1-P_{0},\label{eqn:row-projection}\end{equation}
 where $P_{0}$ is the orthogonal projection onto $\cF_{0}$.

\subsubsection{\label{sub:unitary-conjugation}The unitary conjugation operator}

For every $n\geq1$, let $\Jq_{n}:\cFq_{n}\to\cFq_{n}$ be the operator
which reverses the order of the components in a tensor in $(\bC^{d})^{\otimes n}$,
i.e, $\Jq_{n}$ acts by the formula in Equation (\ref{eqn:1.11})
of the Introduction. A consequence of Equation (\ref{eqn:def-inner-prod}),
which defines the inner product $\langle\cdot,\cdot\rangle_{q}$,
is that $\Jq_{n}$ is a unitary operator in $B(\cFq_{n})$. Indeed,
this is easily seen to follow from Equation (\ref{eqn:def-inner-prod})
and the following basic fact about inversions of permutations: if
$\theta$ denotes the special permutation which reverses the order
on $\{1,\ldots,n\}$, then one has $\mbox{inv}(\theta\tau\theta)=\mbox{inv}(\tau)$
for every permutation $\tau$ of $\{1,\ldots,n\}$.

Therefore, we can speak of the unitary operator $\Jq\in B(\cFq)$
from Definition \ref{def:conjugation-operator}, which is obtained
as $\Jq:=\oplus_{n=0}^{\infty}\Jq_{n}$. Note that $\Jq$ is an involution,
i.e. $(\Jq)^{2}=I_{\cFq}$, and that it intertwines the left and right
creation operators, i.e. \begin{equation}
\Rq_{j}=\Jq\Lq_{j}\Jq,\quad1\leq j\leq d.\label{eqn:J-op-affects-Li}\end{equation}

\subsection{\label{sub:original-unitary}The original unitary operator}

In this subsection, we review the construction of the unitary $U:\cFq\to\cF$
from \cite{DN1993}, which appears in Definition \ref{def:Uopp}.
An important role in the construction of this unitary is played by
the positive operator \[
\Mq:=\sum_{j=1}^{d}\Lq_{j}(\Lq_{j})^{*}\in B(\cFq).\]
 Clearly $\Mq$ can be written as a direct sum $\Mq=\oplus_{n=0}^{\infty}\Mq_{n}$,
where $\Mq_{n}$ is a positive operator on $\cFq_{n}$, for every
$n\geq0$. Using Equation (\ref{eqn:creation}) and Equation (\ref{eqn:annihilation}),
one can show that $M_{n}^{(q)}$ acts on the natural basis of $\cFq_{n}$
by \begin{equation}
\Mq_{n}(\xi_{i_{1}}\otimes\cdots\otimes\xi_{i_{n}})=\sum_{m=1}^{n}q^{m-1}\xi_{i_{m}}\otimes\xi_{i_{1}}\otimes\cdots\otimes\widehat{\xi_{i_{m}}}\otimes\cdots\otimes\xi_{i_{n}}.\label{eqn:how-Mq-acts}\end{equation}
 (Recall that the {}``hat'' symbol over the component $\xi_{i_{m}}$
means that it is deleted from the tensor.)

With the exception of $\Mq_{0}$ (which is zero), the operators $\Mq_{n}$
are invertible. This is implied by Lemma 4.1 of \cite{DN1993}, which
also gives the estimate \begin{equation}
\|(\Mq_{n})^{-1}\|\leq(1-|q|)\prod_{k=1}^{\infty}\frac{1+|q|^{k}}{1-|q|^{k}}<\infty,\quad\forall n\geq1.\label{eqn:DNspectrum}\end{equation}
 An important thing to note about Equation (\ref{eqn:DNspectrum})
is that the upper bound on the right-hand side is independent of $n$.

The unitary operator $U$ is defined as a direct sum, $U:=\oplus_{n=0}^{\infty}U_{n}$,
where the unitaries $U_{n}:\cFq_{n}\to\cF_{n}$ are defined recursively
as follows: we first define $U_{0}$ by $U_{0}(\Omega)=\Omega$, and
for every $n\geq1$ we define $U_{n}$ by \begin{equation}
U_{n}:=(I\otimes U_{n-1})(\Mq_{n})^{1/2}.\label{eqn:def-Un}\end{equation}
 In Proposition 3.2 of \cite{DN1993} it was shown that $U_{n}$ as
defined in Equation (\ref{eqn:def-Un}) is actually a unitary operator,
and hence that $U$ is a unitary operator. Moreover, in Section 4
of \cite{DN1993} it was shown that $\cC\subseteq U\cCq U^{*}$ for
every $q\in(-1,1)$.

\subsection{\label{sub:sbl-operators}Summable band-limited operators}

Throughout this section, we fix a Hilbert space $\cH$, and in addition
we fix an orthogonal direct sum decomposition of $\cH$ as \begin{equation}
\cH=\oplus_{n=0}^{\infty}\cH_{n}.\label{eqn:orthog-decomp-H}\end{equation}
 We will study certain properties an operator $T\in B(\cH)$ can have
with respect to this decomposition of $\cH$. We would like to emphasize
that the concepts considered here depend not only on $\cH$, but also
on the orthogonal decomposition for $\cH$ in Equation (\ref{eqn:orthog-decomp-H}). 
\begin{defn}
\noindent \label{def:band-limited} Let $T$ be an operator in $B(\cH)$.
If there exists a non-negative integer $b$ such that \begin{equation}
T(\cH_{n})\subseteq\bigoplus_{\substack{m\geq0\\
|m-n|\leq b}
}\cH_{m},\quad\forall n\geq0,\label{eqn:band-limited}\end{equation}
 then we will say that $T$ is \emph{band-limited}. A number $b$
as in Equation (\ref{eqn:band-limited}) will be called a \emph{band
limit} for $T$. The set of all band-limited operators in $B(\cH)$
will be denoted by $\cB$. 
\end{defn}

\begin{defn}
\noindent \label{def:summable-band-limited} Let $T$ be an operator
in $\cB$. We will say that $T$ is \emph{summable} when it has the
property that \[
\sum_{n=0}^{\infty}\|T\mid{}_{\cH_{n}}\|<\infty,\]
 where we have used $T\mid{}_{\cH_{n}}\in B(\cH_{n},\cH)$ to denote
the restriction of $T$ to $\cH_{n}$. The set of all summable band-limited
operators in $B(\cH)$ will be denoted by $\cS$.\end{defn}
\begin{prop}
With respect to the preceding definitions, 
\begin{enumerate}
\item $\cB$ is a unital $*$-subalgebra of $B(\cH)$ and 
\item $\cS$ is a two-sided ideal of $\cB$ which is closed under taking
adjoints. 
\end{enumerate}
\end{prop}
\begin{proof}
The proof of (1) is left as an easy exercise for the reader. To verify
(2), we first show that $\cS$ is closed under taking adjoints. Suppose
$T\in\cS$, and let $b$ be a band limit for $T$. By examining the
matrix representations of $T$ and of $T^{*}$ with respect to the
orthogonal decomposition (\ref{eqn:orthog-decomp-H}), it is easily
verified that \[
\|T^{*}\mid_{\cH_{n}}\|\leq\sum_{\substack{m\geq0\\
|m-n|\leq b}
}\|T\mid_{\cH_{m}}\|,\quad\forall n\geq0.\]
 This implies that \[
\sum_{n=0}^{\infty}\|T^{*}\mid_{\cH_{n}}\|\leq(2b+1)\sum_{m=0}^{\infty}\|T\mid_{\cH_{m}}\|<\infty,\]
 which gives $T^{*}\in\cS$. Next, we show that $\cS$ is a two-sided
ideal of $\cB$. Since $\cS$ was proved to be self-adjoint, it will
suffice to show that it is a left ideal. It is clear that $\cS$ is
closed under linear combinations. The fact that $\cS$ is a left ideal
now follows from the simple observation that for $T\in\cB$ and $S\in\cS$
we have \[
\sum_{n=0}^{\infty}\|TS\mid{}_{\cH_{n}}\|\leq\|T\|\sum_{n=0}^{\infty}\|S\mid{}_{\cH_{n}}\|,<\infty,\]
 which implies $TS\in\cS$. 
\end{proof}
In the following definition, we identify some special types of band-limited
operators.
\begin{defn}
\label{def:types-of-band-ops} Let $T$ be an operator in $\cB$. 
\begin{enumerate}
\item If $T$ satisfies $T(\cH_{n})\subseteq\cH_{n}$ for all $n\geq0$,
then we will say that $T$ is \emph{block-diagonal}. 
\item If there is $k\geq0$ such that $T$ satisfies $T(\cH_{n})\subseteq\cH_{n+k}$
for $n\geq0$, then we will say that $T$ is \emph{$k$-raising.} 
\item If there is $k\geq0$ such that $T$ satisfies $T(\cH_{n})\subseteq\cH_{n-k}$
for $n\geq k$ and $T(\cH_{n})=\{0\}$ for $n<k$, then we will say
that $T$ \emph{is $k$-lowering.} 
\end{enumerate}
\end{defn}
Note that a block-diagonal operator is both $0$-raising and $0$-lowering.

The following proposition gives a Fourier-type decomposition for band-limited
operators.
\begin{prop}
\label{prop:fourier-decomp} Let $T$ be an operator in $\cB$ with
a band-limit $b\geq0$, as in Definition \ref{def:band-limited}.
Then we can decompose $T$ as \begin{equation}
T=\sum_{k=0}^{b}X_{k}+\sum_{k=1}^{b}Y_{k},\label{eqn:fourier-decomp}\end{equation}
 where each $X_{k}$ is a $k$-raising operator for $0\leq k\leq b$,
and each $Y_{k}$ is a $k$-lowering operator for $1\leq k\leq b$.
This decomposition is unique. Moreover, if $T$ is summable in the
sense of Definition \ref{def:summable-band-limited}, then each of
the $X_{k}$ and $Y_{k}$ are summable. \end{prop}
\begin{proof}
First, fix an integer $k$ satisfying $0\leq k\leq b$. For each $n\geq0$,
consider the linear operator $P_{n+k}T\mid_{\cH_{n}}\in B(\cH_{n},\cH_{n+k})$
which results from composing the orthogonal projection $P_{n+k}$
onto $\cH_{n+k}$ with the restriction $T\mid_{\cH_{n}}.$ Clearly
$\|P_{n+k}T\mid_{\cH_{n}}\|\leq\|T\|$. This allows us to define an
operator $X_{k}\in B(\cH)$ which acts on $\cH_{n}$ by \begin{equation}
X_{k}\xi=P_{n+k}T\xi,\quad\forall\xi\in\cH_{n}.\label{eq:Xs}\end{equation}
 It follows from this definition that $X_{k}$ is a $k$-raising operator.

Similarly, for an integer $k$ satisfying $1\leq k\leq b$, we can
define a $k$-lowering operator $Y_{k}\in B(\cH)$ which acts on $\mbox{\ensuremath{\xi\in}}\cH_{n}$
by\begin{equation}
Y_{k}\xi=\begin{cases}
P_{n-k}T\xi & \mbox{if }k\leq n,\\
0 & \mbox{if }k>n.\end{cases}\label{eq:Ys}\end{equation}

It's clear that Equation (\ref{eqn:fourier-decomp}) holds with each
$X_{k}$ and $Y_{k}$ defined as above. Conversely, if Equation (\ref{eqn:fourier-decomp})
holds, then it's clear that each $X_{k}$ and $Y_{k}$ is completely
determined as in Equation (\ref{eq:Xs}) and Equation (\ref{eq:Ys})
respectively. This implies the uniqueness of this decomposition.

Finally, suppose $T$ is summable. The fact that each $X_{k}$ and
$Y_{k}$ is summable then follows from the observation that Equation
(\ref{eq:Xs}) and Equation (\ref{eq:Ys}) imply $\|X_{k}\mid_{\cH_{n}}\|\leq\|T\mid_{\cH_{n}}\|$
and $\|Y_{k}\mid_{\cH_{n}}\|\leq\|T\mid_{\cH_{n}}\|$ for every $n\geq0$. 
\end{proof}
The following result about commutators will be needed in Section \ref{sec:main-result}. 
\begin{prop}
\label{prop:comm-using-pedersen} Let $T\in\cB$ be a positive block-diagonal
operator, and let $V\in\cB$ be a $1$-raising operator. Suppose that
the commutator $[T,V]$ satisfies \begin{equation}
\sum_{n=0}^{\infty}\|[T,V]\mid{}_{\cH_{n}}\|^{1/2}<\infty.\label{eqn:sqrt-summable}\end{equation}
 Then the commutator $[T^{1/2},V]$ is a summable 1-raising operator.\end{prop}
\begin{proof}
For every $n\geq0$, let $T_{n}=T\mid_{\cH_{n}}\in B(\cH_{n})$ and
let $V_{n}=V\mid_{\cH_{n}}\in B(\cH_{n},\cH_{n+1})$. Since $T$ is
block-diagonal and $V$ is $1$-raising, it's clear that $[T,V]$
and $[T^{1/2},V]$ are 1-raising operators which satisfy \[
[T,V]\mid{}_{\cH_{n}}=T_{n+1}V_{n}-V_{n}T_{n},\quad\forall n\geq0,\]
 and\[
[T^{1/2},V]\mid{}_{\cH_{n}}=T_{n+1}^{1/2}V_{n}-V_{n}T_{n}^{1/2},\quad\forall n\geq0.\]
 It follows that the hypothesis (\ref{eqn:sqrt-summable}) can be
rewritten as\[
\sum_{n=0}^{\infty}\|T_{n+1}V_{n}-V_{n}T_{n}\|^{1/2}<\infty,\]
 while the required conclusion that $[T^{1/2},V]\in\cS$ is equivalent
to \[
\sum_{n=0}^{\infty}\|T_{n+1}^{1/2}V_{n}-V_{n}T_{n}^{1/2}\|<\infty.\]
 We will prove that this holds by showing that for every $n\geq0$,
\begin{equation}
\|T_{n+1}^{1/2}V_{n}-V_{n}T_{n}^{1/2}\|\leq\frac{5}{4}\|V\|^{1/2}\|T_{n+1}V_{n}-V_{n}T_{n}\|^{1/2}.\label{eqn:pedersen-comm-final-ineq}\end{equation}

For the rest of the proof, fix $n\geq0$. Consider the operators $A,B\in B(\cH_{n}\oplus\cH_{n+1})$
which, written as $2\times2$ matrices, are given by \[
A:=\left[\begin{array}{cc}
T_{n} & 0\\
0 & T_{n+1}\end{array}\right],\qquad B:=\left[\begin{array}{cc}
0 & V_{n}^{*}\\
V_{n} & 0\end{array}\right].\]
 Since $T$ is positive, it follows that $A$ is positive, with \[
A^{1/2}=\left[\begin{array}{cc}
T_{n}^{1/2} & 0\\
0 & T_{n+1}^{1/2}\end{array}\right].\]
 A well-known commutator inequality (see e.g. \cite{P1993}) gives
\begin{equation}
\|[A^{1/2},B]\|\leq\frac{5}{4}\|B\|^{1/2}\|[A,B]\|^{1/2}.\label{eqn:2C.65}\end{equation}
 From the definitions of $A$ and $B$, we compute \[
[A,B]=\left[\begin{array}{cc}
0 & (T_{n+1}V_{n}-V_{n}T_{n})^{*}\\
T_{n+1}V_{n}-V_{n}T_{n} & 0\end{array}\right],\]
 and this implies $\|[A,B]\|=\|T_{n+1}V_{n}-V_{n}T_{n}\|$. Similarly,
$\|[A^{1/2},B]\|=\|T_{n+1}^{1/2}V_{n}-V_{n}T_{n}^{1/2}\|$, and it's
clear that $\|B\|=\|V_{n}\|$. By substituting these equalities into
(\ref{eqn:2C.65}) we obtain \[
\|T_{n+1}^{1/2}V_{n}-V_{n}T_{n}^{1/2}\|\leq\frac{5}{4}\|V_{n}\|^{1/2}\|T_{n+1}V_{n}-V_{n}T_{n}\|^{1/2}.\]
 Since $\|V_{n}\|\leq\|V\|$, this clearly implies that (\ref{eqn:pedersen-comm-final-ineq})
holds. 
\end{proof}

\section{An inclusion criterion}

In this section, we work exclusively in the framework of the (non-deformed)
extended Cuntz algebra $\cC$. We will use the terminology of Subsection
\ref{sub:sbl-operators} with respect to the natural decomposition
$\cF=\oplus_{n=0}^{\infty}\cF_{n}$. In particular, we will refer
to the unital $*$-subalgebra $\cB\subseteq B(\cF)$ which consists
of band-limited operators as in Definition \ref{def:band-limited},
and to the ideal $\cS$ of $\cB$ which consists of summable band-limited
operators as in Definition \ref{def:summable-band-limited}.

The main result of this section is Theorem \ref{thm:inclusion-criterion}.
This is an analogue in the $\mathrm{C}^{*}$-framework of the bicommutant
theorem from von Neumann algebra theory, where we restrict our attention
to the $*$-algebra $\cB$ and consider commutators modulo the ideal
$\cS$. In this framework, the role of ``commutant'' is played
by the $\mathrm{C}^{*}$-algebra generated by right creation operators
on $\cF$.

For clarity, we will first consider the special case of a block-diagonal
operator.
\begin{defn}
\label{def:approximants} Let $T\in\cB$ be a block-diagonal operator.
The sequence of \emph{$\mathcal{C}$-approximants} for $T$ is the
sequence $(A_{n})_{n=0}^{\infty}$ of block-diagonal elements of $\cC$
defined recursively as follows: we first define $A_{0}$ by $A_{0}=\langle T(\Omega),\Omega\rangle I_{\cF}$,
and for every $n\geq0$ we define $A_{n+1}$ by \begin{equation}
A_{n+1}:=A_{n}+\sum_{\begin{array}{c}
{\scriptstyle 1\leq i_{1},\ldots,i_{n+1}\leq d}\\
{\scriptstyle 1\leq j_{1},\ldots,j_{n+1}\leq d}\end{array}}c_{i_{1},\ldots,i_{n+1};j_{1},\ldots,j_{n+1}}\bigl(L_{i_{1}}\cdots L_{i_{n+1}}\bigr)\bigl(L_{j_{1}}\cdots L_{j_{n+1}}\bigr)^{*},\label{eqn:approximant}\end{equation}
 where the coefficients $c_{i_{1},\ldots,i_{n+1};j_{1},\ldots,j_{n+1}}$
are defined by \begin{equation}
c_{i_{1},\ldots,i_{n+1};j_{1},\ldots,j_{n+1}}:=\begin{array}[t]{l}
\langle T(\xi_{j_{1}}\otimes\cdots\otimes\xi_{j_{n+1}}),\xi_{i_{1}}\otimes\cdots\otimes\xi_{i_{+1n}}\rangle\\
-\delta_{i_{n+1},j_{n+1}}\cdot\langle T(\xi_{j_{1}}\otimes\cdots\otimes\xi_{j_{n}}),\xi_{i_{1}}\otimes\cdots\otimes\xi_{i_{n}}\rangle.\end{array}\label{eqn:approximant-coeff}\end{equation}
 
\end{defn}
The main property of the approximant $A_{n}$ is that it agrees with
the operator $T$ on each subspace $\cF_{m}$ for $m\leq n$. More
precisely, we have the following lemma.
\begin{lem}
\label{lem:approximants-work} Let $T\in\cB$ be a block-diagonal
operator, and let $(A_{n})_{n=0}^{\infty}$ be the sequence of $\mathcal{C}$-approximants
for $T$, as in Definition \ref{def:approximants}. Then for every
$m\geq0$, \begin{equation}
A_{n}\mid{}_{\cF_{m}}=\left\{ \begin{array}{ll}
T\mid{}_{\cF_{m}} & \mbox{if }m\leq n,\\
(T\mid{}_{\cF_{n}})\otimes I_{m-n} & \mbox{if }m>n.\end{array}\right.\label{eqn:3.31}\end{equation}
 \end{lem}
\begin{proof}
We will show that for every fixed $n\geq0$, Equation (\ref{eqn:3.31})
holds for all $m\geq0$. The proof of this statment will proceed by
induction on $n$. The base case $n=0$ is left as an easy exercise
for the reader. The remainder of the proof is devoted to the induction
step. Fix $n\geq0$ and assume that Equation (\ref{eqn:3.31}) holds
for this $n$ and for all $m\geq0$. We will prove the analogous statement
for $n+1$.

From Equation (\ref{eqn:approximant}), it is immediate that \[
A_{n+1}\mid_{\cF_{m}}=A_{n}\mid_{\cF_{m}}=T\mid_{\cF_{m}},\quad\forall m\leq n.\]
 Thus it remains to fix $m\geq n+1$ and verify that \[
A_{n+1}\mid{}_{\cF_{m}}=(T\mid{}_{\cF_{n+1}})\otimes I_{m-n-1}\in B(\cF_{m}).\]
 In light of how $(T\mid{}_{\cF_{n+1}})\otimes I_{m-n-1}$ acts on
the canonical basis of $\cF_{m}$, this amounts to showing that for
every $1\leq k_{1},\ldots,k_{m},\ell_{1},\ldots,\ell_{m}\leq d$,
one has\begin{multline}
\langle A_{n+1}(\xi_{\ell_{1}}\otimes\cdots\otimes\xi_{\ell_{m}}),\xi_{k_{1}}\otimes\cdots\otimes\xi_{k_{m}}\rangle\\
=\delta_{k_{n+2},\ell_{n+2}}\cdots\delta_{k_{m},\ell_{m}}\langle T(\xi_{\ell_{1}}\otimes\cdots\otimes\xi_{\ell_{n+1}}),\xi_{k_{1}}\otimes\cdots\otimes\xi_{k_{n+1}}\rangle.\label{eqn:lem-approx-must-show}\end{multline}

On the left-hand side of Equation (\ref{eqn:lem-approx-must-show})
we substitute for $A_{n+1}$ using the recursive definition given
by Equation (\ref{eqn:approximant}). This gives\begin{multline}
\langle A_{n+1}(\xi_{\ell_{1}}\otimes\cdots\otimes\xi_{\ell_{m}}),\xi_{k_{1}}\otimes\cdots\otimes\xi_{k_{m}}\rangle\\
=\langle A_{n}\xi_{l_{1}}\otimes\cdots\otimes\xi_{l_{m}}),\xi_{k_{1}}\otimes\cdots\xi_{k_{m}}\rangle\\
+\sum_{\begin{array}{c}
{\scriptstyle i_{1},\ldots,i_{n+1}}\\
{\scriptstyle j_{1},\ldots,j_{n+1}}\end{array}}c_{i_{1},\ldots,i_{n+1};j_{1},\ldots,j_{n+1}}\alpha(i_{1},\ldots,i_{n+1};j_{1},\ldots,j_{n+1}),\label{eqn:3.60}\end{multline}
where for every $1\leq i_{1},\ldots i_{n+1},j_{1},\ldots,j_{n+1}\leq d$,
we have written\begin{multline*}
\alpha(i_{1},\ldots,i_{n+1};j_{1},\ldots,j_{n+1})\\
=\langle\bigl(L_{i_{1}}\cdots L_{i_{n+1}}\bigr)\bigl(L_{j_{1}}\cdots L_{j_{n+1}}\bigr)^{*}(\xi_{\ell_{1}}\otimes\cdots\otimes\xi_{\ell_{m}}),(\xi_{k_{1}}\otimes\cdots\otimes\xi_{k_{m}})\rangle.\end{multline*}
It is clear that an inner product like the one just written simplifies
as follows: \begin{align*}
\langle & \bigl(L_{i_{1}}\cdots L_{i_{n+1}}\bigr)\bigl(L_{j_{1}}\cdots L_{j_{n+1}}\bigr)^{*}(\xi_{\ell_{1}}\otimes\cdots\otimes\xi_{\ell_{m}}),(\xi_{k_{1}}\otimes\cdots\otimes\xi_{k_{m}})\rangle\\
 & =\langle\bigl(L_{j_{1}}\cdots L_{j_{n+1}}\bigr)^{*}(\xi_{\ell_{1}}\otimes\cdots\otimes\xi_{\ell_{m}}),(L_{i_{1}}\cdots L_{i_{n+1}})^{*}(\xi_{k_{1}}\otimes\cdots\otimes\xi_{k_{m}})\rangle\\
 & =\delta_{i_{1},k_{1}}\cdots\delta_{i_{n+1},k_{n+1}}\delta_{j_{1},\ell_{1}}\cdots\delta_{j_{n+1},\ell_{n+1}}\langle\xi_{\ell_{n+2}}\otimes\cdots\otimes\xi_{\ell_{m}},\xi_{k_{n+2}}\otimes\cdots\otimes\xi_{k_{m}}\rangle\\
 & =\delta_{i_{1},k_{1}}\cdots\delta_{i_{n+1},k_{n+1}}\delta_{j_{1},\ell_{1}}\cdots\delta_{j_{n+1},\ell_{n+1}}\delta_{\ell_{n+2},k_{n+2}}\cdots\delta_{\ell_{m},k_{m}}.\end{align*}
 Thus in the sum on the right-hand side of Equation (\ref{eqn:3.60}),
the only term that survives is the one corresponding to $i_{1}=k_{1},\ldots,i_{n+1}=k_{n+1}$
and $j_{1}=\ell_{1},\ldots,j_{n+1}=\ell_{n+1}$, and we obtain that\begin{multline}
\langle A_{n+1}(\xi_{\ell_{1}}\otimes\cdots\otimes\xi_{\ell_{m}}),\xi_{k_{1}}\otimes\cdots\otimes\xi_{k_{m}}\rangle\\
=\langle A_{n}(\xi_{\ell_{1}}\otimes\cdots\otimes\xi_{\ell_{m}}),\xi_{k_{1}}\otimes\cdots\otimes\xi_{k_{m}}\rangle\\
+\delta_{\ell_{n+2},k_{n+2}}\cdots\delta_{\ell_{m},k_{m}}c_{k_{1},\ldots,k_{n+1};\ell_{1},\ldots,\ell_{n+1}}.\label{eqn:3.70}\end{multline}
Finally, we remember our induction hypothesis, which gives\begin{multline}
\langle A_{n}(\xi_{\ell_{1}}\otimes\cdots\otimes\xi_{\ell_{m}}),\xi_{k_{1}}\otimes\cdots\otimes\xi_{k_{m}}\rangle\\
=\delta_{k_{n+1},\ell_{n+1}}\cdots\delta_{k_{m},\ell_{m}}\langle T(\xi_{\ell_{1}}\otimes\cdots\otimes\xi_{\ell_{n}}),\xi_{k_{1}}\otimes\cdots\otimes\xi_{k_{n}}\rangle.\label{eqn:3.80}\end{multline}
A straightforward calculation shows that if we substitute Equation
(\ref{eqn:3.80}) into Equation (\ref{eqn:3.70}) and use Formula
(\ref{eqn:approximant-coeff}) which defines the coefficient \linebreak$c_{k_{1},\ldots,k_{n+1};\ell_{1},\ldots,\ell_{n+1}}$,
then we arrive at the right-hand side of Equation (\ref{eqn:lem-approx-must-show}).
This completes the induction argument.\end{proof}
\begin{lem}
\label{lem:diff-approximants}Let $T\in\cB$ be a block-diagonal operator,
and let $(A_{n})_{n=1}^{\infty}$ be the sequence of $\mathcal{C}$-approximants
for $T$, as in Definition \ref{def:approximants}. Then for every
$n\geq1$, \begin{equation}
\|A_{n+1}-A_{n}\|=\|T\mid{}_{\cF_{n+1}}-(T\mid{}_{\cF_{n}})\otimes I\|.\label{eqn:diff-approximants}\end{equation}
 \end{lem}
\begin{proof}
Note that since $A_{n+1}-A_{n}$ is block-diagonal, \[
\|A_{n+1}-A_{n}\|=\sup_{m\geq0}\|A_{n+1}\mid_{\cF_{m}}-A_{n}\mid_{\cF_{m}}\|.\]
 To compute this supremum, there are three cases to consider. In each
case we apply Lemma \ref{lem:approximants-work}. First, for $m\leq n$,
\[
\|A_{n+1}\mid_{\cF_{m}}-A_{n}\mid_{\cF_{m}}\|=0.\]
 Next, for $m=n+1$, \[
\|A_{n+1}\mid_{\cF_{n+1}}-A_{n}\mid_{\cF_{n+1}}\|=\|T\mid_{\cF_{n+1}}-(T\mid_{\cF_{n}})\otimes I\|.\]
 Finally, for $m>n+1$, \begin{eqnarray*}
\|A_{n+1}\mid_{\cF_{m}}-A_{n}\mid_{\cF_{m}}\| & = & \|(T\mid_{\cF_{n+1}})\otimes I_{m-n-1}-(T\mid_{\cF_{n}})\otimes I_{m-n}\|\\
 & = & \|(T\mid_{\cF_{n+1}}-(T\mid_{\cF_{n}})\otimes I)\otimes I_{m-n-1}\|\\
 & = & \|T\mid_{\cF_{n+1}}-(T\mid_{\cF_{n}})\otimes I\|.\end{eqnarray*}
 This makes it clear that the supremum over all $m\geq0$ is equal
to the right hand side of Equation (\ref{eqn:diff-approximants}),
as required. \end{proof}
\begin{lem}
\label{lem:diff-approximants-inclusion} Let $T$ be a block-diagonal
operator. If $T$ satisfies \[
\sum_{n=1}^{\infty}\|(T\mid_{\cF_{n+1}})-(T\mid_{\cF_{n}})\otimes I\|<\infty,\]
 then $T\in\cC$. \end{lem}
\begin{proof}
Let $(A_{n})_{n=1}^{\infty}$ be the sequence of $\mathcal{C}$-approximants
for $T$, as in Definition \ref{def:approximants}. In view of Lemma
\ref{lem:diff-approximants}, the hypothesis of the present lemma
implies that the sum $\sum_{n=1}^{\infty}\|A_{n+1}-A_{n}\|$ is finite.
This in turn implies that the sequence $(A_{n})_{n=1}^{\infty}$ converges
in norm to an operator $A$. Since each $A_{n}$ belongs to $\cC$,
it follows that $A$ belongs to $\cC$. But we must have $A=T$, as
Lemma \ref{lem:approximants-work} implies that \[
A\mid_{\cF_{m}}=\lim_{n\to\infty}A_{n}\mid_{\cF_{m}}=T\mid_{\cF_{m}},\ \ \forall\, m\geq0.\]
 Hence $T\in\cC$, as required. \end{proof}
\begin{prop}
\label{pro:inclusion-criterion} Let $T$ be a block-diagonal operator.
If the block-diagonal operator $T-\sum_{i=1}^{d}R_{i}TR_{i}^{*}$
belongs to the ideal $\cS$, then $T\in\cC$. \end{prop}
\begin{proof}
The hypothesis is equivalent to \begin{equation}
\sum_{n=1}^{\infty}\|(T-{\textstyle \sum}_{i=1}^{d}R_{i}TR_{i}^{*})\mid_{\cF_{n}}\|<\infty.\label{eqn:3.51}\end{equation}
 It's easy to verify that for $n\geq1$, \[
({\textstyle \sum}_{i=1}^{d}R_{i}TR_{i}^{*})\mid_{\cF_{n}}=(T\mid_{\cF_{n-1}})\otimes I,\]
 which gives \[
\|(T-{\textstyle \sum}_{i=1}^{d}R_{i}TR_{i}^{*})\mid_{\cF_{n}}\|=\|T\mid_{\cF_{n}}-(T\mid_{\cF_{n-1}})\otimes I\|.\]
 Therefore, (\ref{eqn:3.51}) implies that the hypothesis of Lemma
\ref{lem:diff-approximants-inclusion} holds, and the result follows
by applying the said lemma. \end{proof}
\begin{cor}
\label{cor:inclusion-criterion}Let $T\in\cB$ be a block-diagonal
operator such that $[T,R_{i}^{*}]\in\cS$ for $1\leq i\leq d$. Then
$T\in\cC$.\end{cor}
\begin{proof}
By Proposition \ref{pro:inclusion-criterion}, it suffices to show
that $T-\sum_{i=1}^{d}R_{i}TR_{i}^{*}\in\cS$. We can write \begin{eqnarray*}
T-{\textstyle \sum_{i=1}^{d}}R_{i}TR_{i}^{*} & = & (P_{0}+{\textstyle \sum_{i=1}^{d}}R_{i}R_{i}^{*})T-{\textstyle \sum_{i=1}^{d}}R_{i}TR_{i}^{*}\\
 & = & P_{0}T-{\textstyle \sum_{i=1}^{d}}R_{i}[T,R_{i}^{*}],\end{eqnarray*}
 where $P_{0}$ is the orthogonal projection onto $\cF_{0}$, and
where we have used Equation (\ref{eqn:row-projection}). Since $P_{0}$
and $[T,R_{i}^{*}]$ belong to $\cS$, and since $T$ and $R_{i}$
belong to $\cB$, the result follows from the fact that $\cS$ is
a two-sided ideal of $\cB$. 
\end{proof}
We now apply the above results on block-diagonal operators in order
to bootstrap the case of general band-limited operators. It is convenient
to first consider the case of $k$-raising/lowering operators, which
were introduced in Definition \ref{def:types-of-band-ops}.
\begin{prop}
\label{prop:inclusion-criterion-for-raising-lowering} Let $T\in\cB$
be a $k$-raising or $k$-lowering operator for some $k\geq0$. If
$T$ satisfies \textup{$[T,R_{j}^{*}]\in\cS$ for $1\leq j\leq d$,
then $T\in\cS$.} \end{prop}
\begin{proof}
First, suppose that $T$ is $k$-raising. For every $1\leq i_{1},\ldots,i_{k}\leq d$,
the fact that the left and right annihilation operators commute implies
that \[
[(L_{i_{1}}\dots L_{i_{k}})^{*}T,R_{j}^{*}]=(L_{i_{1}}\dots L_{i_{k}})^{*}[T,R_{j}^{*}],\quad\forall1\leq j\leq d.\]
 Since $[T,R_{j}^{*}]\in\cS$ by hypothesis, and since $\cS$ is a
two-sided ideal of $\cB$, it follows that $[(L_{i_{1}}\dots L_{i_{k}})^{*}T,R_{j}^{*}]\in\cS$.
The operator $(L_{i_{1}}\dots L_{i_{k}})^{*}T$ is block-diagonal,
hence Corollary \ref{cor:inclusion-criterion} gives $(L_{i_{1}}\dots L_{i_{k}})^{*}T\in\cC$.

Since $T$ is $k$-raising, the range of $T$ is orthogonal to the
subspace $\cF_{\ell}$ whenever $\ell<k$. This implies that \[
\Bigl(I-\sum_{1\leq i_{1},\ldots,i_{k}\leq d}L_{i_{1}}\cdots L_{i_{k}}(L_{i_{1}}\cdots L_{i_{k}})^{*}\Bigr)T=0.\]
 Hence \[
T=\sum_{1\leq i_{1},\ldots,i_{k}\leq d}L_{i_{1}}\cdots L_{i_{k}}\bigl((L_{i_{1}}\cdots L_{i_{k}})^{*}T\bigr),\]
 and it follows that $T\in\cC$.

The case when $T$ is $k$-lowering is handled in a similar way by
considering the operators $TL_{i_{1}}\dots L_{i_{k}}$ for every $1\leq i_{1},\ldots,i_{k}\leq d$. \end{proof}
\begin{thm}
\label{thm:inclusion-criterion} Let $T\in\cB$ be an operator such
that either $[T,R_{j}^{*}]\in\cS$ for all $1\leq j\leq d$, or $[T,R_{j}]\in\cS$
for all $1\leq j\leq d$. Then $T\in\cC$. \end{thm}
\begin{proof}
First, suppose that $T$ satisfies $[T,R_{j}^{*}]\in\cS$ for every
$1\leq j\leq d$. Let $b\geq0$ be a band-limit for $T$. By Proposition
\ref{prop:fourier-decomp}, we can decompose $T$ as \[
T=\sum_{k=0}^{b}X_{k}+\sum_{k=1}^{b}Y_{k},\]
 where each $X_{k}$ is a $k$-raising operator, and each $Y_{k}$
is a $k$-lowering operator. We will prove that each $X_{k}\in\cC$
and each $Y_{k}\in\cC$.

Fix for the moment $1\leq j\le d$. We have \begin{align}
[T,R_{j}^{*}] & =\sum_{k=0}^{b}[X_{k},R_{j}^{*}]+\sum_{k=1}^{b}[Y_{k},R_{j}^{*}]\nonumber \\
 & =\sum_{k=0}^{b+1}X_{k}'+\sum_{k=0}^{b+1}Y_{k}',\label{eqn:3.81}\end{align}
 where \[
X_{k}'=\begin{cases}
[X_{k+1},R_{j}^{*}] & \mbox{if }0\leq k\leq b-1,\\
0 & \mbox{if }k=b\mbox{ or }k=b+1,\end{cases}\]
 and\[
Y_{k}'=\begin{cases}
[X_{0},R_{j}^{*}] & \mbox{if }k=1,\\
{}[Y_{k-1},R_{j}^{*}] & \mbox{if }2\leq k\leq b+1.\end{cases}\]
 It is clear that each $X_{k}'$ is a $k$-raising operator, and that
each $Y_{k}'$ is a $k$-lowering operator. Hence Equation (\ref{eqn:3.81})
provides the (unique) Fourier-type decomposition for $[T,R_{j}^{*}]$,
as in Proposition \ref{prop:fourier-decomp}. Since it is given that $[T,R_{j}^{*}]\in\cS$,
Proposition \ref{prop:fourier-decomp} implies that each $X_{k}'\in\cS$
and each $Y_{k}'\in\cS$. This in turn implies that $[X_{k},R_{j}^{*}]\in\cS$
for every $0\leq k\leq b$, and that $[Y_{k},R_{j}^{*}]\in\cS$ for
every $1\leq k\leq b$.

Now let us unfix the index $j$ from the preceding paragraph. For
every $0\leq k\leq b$, we have proved that $[X_{k},R_{j}^{*}]\in\cS$
for all $1\leq j\leq d$, hence Proposition \ref{prop:inclusion-criterion-for-raising-lowering}
implies that $X_{k}\in\cC$. The fact that $Y_{k}\in\cC$ for every
$1\leq k\leq b$ is obtained in the same way. This concludes the proof
in the case when the hypothesis on $T$ is that $[T,R_{j}^{*}]\in\cS$
for all $1\leq j\leq d$.

If $T$ satisfies $[T,R_{j}]\in\cS$ for all $1\leq j\leq d$, then
since the ideal $\cS$ is closed under taking adjoints, it follows
that $[T^{*},R_{j}^{*}]\in\cS$ for all $1\leq j\leq d$. The above
arguments therefore apply to $T^{*}$, and lead to the conclusion
that $T^{*}\in\cC$, which gives $T\in\cC$. 
\end{proof}

\section{\label{sec:main-result}Construction of the embedding}

In this section we fix a deformation parameter $q\in(-1,1)$ and consider
the $C^{*}$-algebra $\cCq=C^{*}(\Lq_{1},\ldots,\Lq_{d})\subseteq B(\cFq)$
from Equation (\ref{eqn:deformed-cuntz}). The main result of this
section (and also this paper), Theorem \ref{thm:main-inclusion},
shows that it is possible to unitarily embed $\cCq$ into the $\mathrm{C}^{*}$-algebra
$\cC=\mathrm{C}^{*}(L_{1},\ldots,L_{d})\subseteq B(\cF)$ from Equation
\ref{eqn:non-deformed-cuntz}.

We will once again utilize the terminology of Subsection \ref{sub:sbl-operators}
with respect to the natural decomposition $\cF=\oplus_{n=0}^{\infty}\cF_{n}$.
In particular, we will refer to the unital $*$-algebra $\cB\subseteq B(\cF)$
consisting of band-limited operators, and to the ideal $\cS$ of $\cB$
consisting of summable band-limited operators.

The deformed Fock space $\cFq$ also has a natural decomposition $\cFq=\oplus_{n=0}^{\infty}\cFq_{n}$,
and we will also need to utilize the terminology of Subsection \ref{sub:sbl-operators}
with respect to this decomposition. We will let $\cBq\subseteq B(\cFq)$
denote the unital $*$-algebra consisting of band-limited operators,
and we will let $\cSq$ denote the ideal of $\cBq$ which consists
of summable band-limited operators.
\begin{rem}
Recall the positive block-diagonal operator $\Mq=\oplus_{n=0}^{\infty}\Mq_{n}\in\cBq$,
which was reviewed in Subsection \ref{sub:original-unitary}. It was recorded
there that for $n\geq1$, $\Mq_{n}$ is an invertible operator on
$\cFq_{n}$. Moreover, for every $n\geq1$, one has the
upper bound (\ref{eqn:DNspectrum}) for the norm $\|(\Mq_{n})^{-1}\|$,
and this upper bound is independent of $n$.

Therefore, the only obstruction to the operator $\Mq$ being invertible
on $\cFq$ is the fact that $\Mq_{0}=0$. We can overcome this obstruction
by working instead with the operator $\Mqplus$ defined by \begin{equation}
\Mqplus:=\Pq_{0}+\Mq,\label{eq:def-Mqplus}\end{equation}
 where $\Pq_{0}\in B(\cFq)$ is the orthogonal projection onto the
subspace $\cFq_{0}$. It's clear that $\Mqplus$ is invertible, and
that the bound from (\ref{eqn:DNspectrum}) applies to $\|(\Mqplus)^{-1}\|$. \end{rem}
\begin{lem}
\label{lem:comm-invsqrtMq-Rqj-est}The operator $\Mqplus$ satisfies
$[(\Mqplus)^{-1/2},\Rq_{j}]\in\cSq$ for all $1\leq j\leq d$. \end{lem}
\begin{proof}
First, we will show that $\Mqplus$ and $\Rq$ satisfy the hypotheses
of Proposition \ref{prop:comm-using-pedersen}. It's clear that $\Mqplus$
is block-diagonal and that $\Rq$ is $1$-raising, but it will require
a bit of work to check that \begin{equation}
\sum_{n=0}^{\infty}\|[\Mqplus,\Rq_{j}]\mid_{\cFq_{n}}\|^{1/2}<\infty,\quad\forall1\leq j\leq d.\label{eq:comm-Mq-Rj-est}\end{equation}

In order to show that (\ref{eq:comm-Mq-Rj-est}) holds, fix $1\leq j\leq d$.
Using Equation (\ref{eq:def-Mqplus}), which defines $\Mqplus$, we
can write \begin{eqnarray*}
[\Mqplus,\Rq_{j}] & = & [\Pq_{0},\Rq]+\sum_{i=1}^{d}[\Lq_{i}(\Lq_{i})^{*},\Rq_{j}]\\
 & = & [\Pq_{0},\Rq]+\sum_{i=1}^{d}\Lq_{i}[(\Lq_{i})^{*},\Rq_{j}],\end{eqnarray*}
 where the last equality follows from the fact that $\Lq_{i}$ and
$\Rq_{j}$ commute. The sum in this equation has only a single non-zero
term. Indeed, as a consequence of Equation (\ref{eqn:comm-relns}),
we have $[(\Lq_{i})^{*},\Rq_{j}]=0$ whenever $i\neq j$. Thus we
arrive at the following formula: \begin{equation}
[\Mqplus,\Rq_{j}]=[\Pq_{0},\Rq]+\Lq_{j}[(\Lq_{j})^{*},\Rq_{j}].\label{eqn:4.20}\end{equation}
 We next restrict the operators on both sides of (\ref{eqn:4.20})
to a subspace $\cFq_{n}$, for $n\geq1$. Noting that $[\Pq_{0},\Rq_{j}]=-\Rq_{j}\Pq_{0}$
vanishes on $\cFq_{n}$, we obtain that \begin{equation}
[\Mqplus,\Rq_{j}]\mid_{\cFq_{n}}=\Lq_{j}[(\Lq_{j})^{*},\Rq_{j}]\mid_{\cFq_{n}},\quad\forall n\geq1.\label{eqn:4.21}\end{equation}
 Finally, we take norms in Equation (\ref{eqn:4.21}) and invoke Equation
(\ref{eqn:comm-relns}) once more to obtain that \[
\|[\Mqplus,\Rq_{j}]\mid_{\cFq_{n}}\|\leq|q|^{n}\,\|\Lq_{j}\|,\quad\forall n\geq1.\]
 The conclusion that (\ref{eq:comm-Mq-Rj-est}) holds follows from
here, since $\sum_{n=1}^{\infty}|q|^{n/2}<\infty$.

Therefore, we can apply Proposition \ref{prop:comm-using-pedersen}
to $\Mqplus$ and $\Rq_{j}$, and conclude that $[(\Mqplus)^{1/2},\Rq_{j}]\in\cSq$.
Note that the operator $(\Mqplus)^{-1/2}$ is bounded and block-diagonal,
meaning in particular that it belongs to the $*$-algebra $\cBq$.
The desired result now follows from the obvious identity \[
[(\Mqplus)^{-1/2},\Rq_{j}]=-(\Mqplus)^{-1/2}[(\Mqplus)^{1/2},\Rq_{j}](\Mqplus)^{-1/2},\]
 and the fact that $\cSq$ is a two-sided ideal of $\cBq$. \end{proof}
\begin{lem}
\label{lem:U-satisfies} For $1\leq j\leq d$, the unitary $U=\oplus_{n=0}^{\infty}U_{n}$
from Subsection \ref{sub:original-unitary} satisfies \begin{equation}
U_{n-1}^{*}L_{j}^{*}U_{n}=(\Lq_{j})^{*}(\Mq_{n})^{-1/2},\quad\forall n\geq1.\label{eqn:lem-Un-satisfies}\end{equation}
 (Note that on the left-hand side of Equation (\ref{eqn:lem-Un-satisfies}),
we view $L_{j}^{*}$ as an operator in $B(\cF_{n},\cF_{n-1})$. On
the right-hand side of Equation (\ref{eqn:lem-Un-satisfies}), we
view $(\Lq_{j})^{*}$ as an operator in $B(\cFq_{n},\cFq_{n-1})$.) \end{lem}
\begin{proof}
Consider the operator $\Aq_{j}:\cFq_{n}\to\cFq_{n-1}$ which acts
on the natural basis of $\cFq_{n}$ by \[
\Aq_{j}(\xi_{i_{1}}\otimes\cdots\otimes\xi_{i_{n}})=\delta_{j,i_{1}}\xi_{i_{2}}\otimes\cdots\otimes\xi_{i_{n}},\quad\forall1\leq i_{1},\ldots,i_{n}\leq d.\]
 We claim that $\Aq_{j}$ satisfies \begin{equation}
\Aq_{j}=(\Lq_{j})^{*}\,(\Mq_{n})^{-1}\label{eq:temp-op-Aj-sat}\end{equation}
 To see this, note that for $1\leq i_{1},\ldots,i_{n}\leq d$, \begin{eqnarray*}
\Aq_{j}\Mq_{n}(\xi_{i_{1}}\otimes\cdots\otimes\xi_{i_{n}}) & = & \Aq_{j}\sum_{m=1}^{n}q^{m-1}\xi_{i_{m}}\otimes\xi_{i_{1}}\otimes\cdots\otimes\widehat{\xi_{i_{m}}}\otimes\cdots\otimes\xi_{i_{n}}\\
 & = & \sum_{m=1}^{n-1}q^{m-1}\delta_{j,i_{m}}\xi_{i_{1}}\otimes\cdots\otimes\widehat{\xi_{i_{m}}}\otimes\cdots\otimes\xi_{i_{n}}\\
 & = & (\Lq_{j})^{*}(\xi_{i_{1}}\otimes\cdots\otimes\xi_{i_{n}}),\end{eqnarray*}
 where the first and last equalities follow from Equation (\ref{eqn:how-Mq-acts})
and Equation (\ref{eqn:annihilation}) respectively. Hence $\Aq_{j}\Mq_{n}=(\Lq_{j})^{*}\mid_{\cFq_{n}}$,
so multiplying on the right by $(\Mq_{n})^{-1}$ establishes the claim.

Now, from Equation (\ref{eqn:def-Un}), which defines $U_{n}$, we
see that \[
U_{n-1}^{*}L_{j}^{*}U_{n}=U_{n-1}^{*}L_{j}^{*}(I\otimes U_{n-1})(\Mq_{n})^{1/2},\]
 and from the definition of $\Aq_{j}$ it's immediate that \[
L_{j}^{*}(I\otimes U_{n-1})=U_{n-1}\Aq_{j}.\]
 Together, this allows us to write \begin{eqnarray*}
U_{n-1}^{*}L_{j}^{*}U_{n} & = & U_{n-1}^{*}U_{n-1}\Aq_{j}(\Mq_{n})^{1/2}\\
 & = & \Aq_{j}(\Mq_{n})^{1/2}.\end{eqnarray*}
 Applying Equation (\ref{eq:temp-op-Aj-sat}) now gives Equation (\ref{eqn:lem-Un-satisfies}),
as required.\end{proof}
\begin{prop}
\label{prop:comm-ULUR} For $1\leq i,j\leq d$, the unitary $U$ from
Subsection \ref{sub:original-unitary} satisfies $[U^{*}L_{j}^{*}U,\Rq_{i}]\in\cSq$.\end{prop}
\begin{proof}
Fix $i$ and $j$ and let $C$ denote the commutator $C=[U^{*}L_{j}^{*}U,\Rq_{i}]$.
It's clear that $C$ is a block-diagonal operator on $\cFq$. In order
to show that $C\in\cSq$, we will need to estimate the norm of its
diagonal blocks.

For $n\geq1$, Lemma \ref{lem:U-satisfies} gives\begin{eqnarray*}
C\mid_{\cFq_{n}} & = & U_{n}^{*}L_{j}^{*}U_{n+1}\Rq_{i}-\Rq_{i}U_{n-1}^{*}L_{j}^{*}U_{n}\\
 & = & (\Lq_{j})^{*}(\Mq_{n+1})^{-1/2}\Rq_{i}-\Rq_{i}(\Lq_{j})^{*}(\Mq_{n})^{-1/2}\\
 & = & (\Lq_{j})^{*}(((\Mq_{n+1})^{-1/2}\Rq_{i}-\Rq_{i}(\Mq_{n})^{-1/2})\\
 &  & +((\Lq_{j})^{*}\Rq_{i}-\Rq_{i}(\Lq_{j})^{*})(\Mq_{n})^{-1/2}.\end{eqnarray*}
 Since $C$ is block-diagonal, this gives\[
C=(\Lq_{j})^{*}[(\Mqplus)^{-1/2},\Rq_{i}]+[(\Lq_{j})^{*},\Rq_{i}](\Mqplus)^{-1/2}.\]
 Now, $[(\Mqplus)^{-1/2},\Rq_{i}]\in\cSq$ by Lemma \ref{lem:comm-invsqrtMq-Rqj-est}.
By Equation (\ref{eqn:comm-relns}),\[
[(\Lq_{j})^{*},\Rq_{i}]\mid_{\cFq_{n}}=\delta_{ij}q^{n}I_{\cFq_{n}},\]
 and since the operator $[(\Lq_{j})^{*},\Rq_{i}]$ is block-diagonal,
this implies that it also belongs to $\cSq$. Since $(\Lq_{j})^{*}$
and $(\Mqplus)^{-1/2}$ both belong to $\cBq$, and since $\cSq$
is a two-sided ideal of $\cBq$, it follows that $C\in\cSq$. 
\end{proof}
We are now able to complete the proof of the embedding theorem. 
\begin{proof}
[Proof of Theorem 1.3]It suffices to show that $\Uopp\Lq_{i}\Uopp^{*}\in\cC$,
for $1\leq i\leq d$. Since $\Uopp\Lq_{i}\Uopp^{*}$ belongs to the
algebra $\cB$ of all band-limited operators, by Theorem \ref{thm:inclusion-criterion}
it will actually be sufficient to verify that\[
[\Uopp\Lq_{i}\Uopp^{*},R_{j}^{*}]\in\cS,\quad\forall1\leq i,j\leq d.\]
 By Definition \ref{def:conjugation-operator}, we can write \begin{eqnarray*}
\Uopp\Lq_{i}\Uopp^{*} & = & JU\Jq\Lq_{i}\Jq U^{*}J\\
 & = & JU\Rq_{i}U^{*}J,\end{eqnarray*}
 where the last equality follows from Equation (\ref{eqn:J-op-affects-Li}).
This gives\begin{eqnarray*}
[\Uopp\Lq_{i}\Uopp^{*},R_{j}^{*}] & = & [JU\Rq_{i}U^{*}J,R_{j}^{*}]\\
 & = & JU[\Rq_{i},U^{*}JR_{j}^{*}JU]U^{*}J\\
 & = & JU[\Rq_{i},U^{*}L_{j}^{*}U](JU)^{*},\end{eqnarray*}
 and we know from Proposition \ref{prop:comm-ULUR} that $[\Rq_{i},U^{*}L_{j}^{*}U]\in\cSq$.
It is clear that conjugation by the unitary $JU$ takes $\cSq$ onto
$\cS$, so this gives the desired result. 
\end{proof}
The proof that $\cCq$ is exact now follows from some simple observations
about nuclear and exact $\mathrm{C}^{*}$-algebras (see e.g. \cite{BO2008}).
\begin{proof}
[Proof of Corollary 1.4]The extended Cuntz algebra $\cC$ is (isomorphic
to) an extension of the Cuntz algebra. Since the Cuntz algebra is
nuclear, this implies that $\cC$ is nuclear, and in particular that
$\cC$ is exact. Since exactness is inherited by subalgebras (see
e.g. Chapter 2 of \cite{BO2008}), it follows from Theorem \ref{thm:main-inclusion}
that $\Uopp\cCq\Uopp^{*}$ is exact, and hence that $\cCq$ is exact. \end{proof}
\begin{rem}
Since Theorem \ref{thm:main-inclusion} holds for all $q\in(-1,1)$,
a natural thought is that the methods used above could also be applied
to establish the inclusion $U\cCq U^{*}\subseteq\cC$ for all $q\in(-1,1)$,
and hence (since the opposite inclusion was shown in \cite{DN1993})
that $U\cCq U^{*}=\cC.$ To do this, it would be necessary to establish
that \begin{equation}
[U\Lq U^{*},R_{j}^{*}]\in\cS,\quad\forall1\leq i,j\leq d.\label{eq:want-to-est}\end{equation}
 This condition looks superficially similar to the condition from
Proposition \ref{prop:comm-ULUR}, but this is deceptive. We believe
that establishing (\ref{eq:want-to-est}) will require a deeper understanding
of the combinatorics which underlie the $q$-commutation relations.

The algebra $\cCq$ arises as a representation of the the univeral
algebra $\cEq$ corresponding to the $q$-commutation relations. It
was shown in \cite{JSW1993} that for $|q|<\sqrt{2}-1$, $\cCq$ and
$\cEq$ are isomorphic (and in particular that they are both isomorphic
to the extended Cuntz algebra). It is believed that this is the case
for all $q\in(-1,1)$. 
\end{rem}

\section{An application to the $q$-Gaussian von Neumann algebras}

The $q$-Gaussian von Neumann algebra $\cMq$ is the von Neumann algebra
generated by $\{\Lq_{i}+(\Lq_{i})^{*}\mid1\leq i\leq d\}$. This algebra
can be considered as a type of deformation of $L(\mathbb{F}_{d})$,
the von Neumann algebra of the free group on $d$ generators. Indeed,
for $q=0$, a basic result in free probability states that $\cMq$
is precisely the realization of $L(\mathbb{F}_{d})$ as the von Neumann
algebra generated by a free semicircular family (see e.g. Section
2.6 of \cite{VDN1992} for the details).

For general $q\in(-1,1)$ it is known that $\cMq$ is a von Neumann
algebra in standard form, with $\Omega$ being a cyclic and separating
trace-vector. The commutant of $\cMq$ is the von Neumann algebra
generated by $\{\Rq_{i}+(\Rq_{i})^{*}\mid1\leq i\leq d\}$ (see Section
2 of \cite{BKS1997}).

Not much is known about the isomorphism class of the algebras $\cMq$
for $q\neq0$. The major open problem is to determine the extent to
which they behave like $L(\mathbb{F}_{d})$. The best results to date
show that $\cMq$ does share certain properties with $L(\mathbb{F}_{d})$.
Nou showed in \cite{N2004} that $\cMq$ is non-injective, and Ricard
showed in \cite{R2005} that it is a $II_{1}$ factor. Shlyakhtenko
showed in \cite{S2004} that if we assume $|q|<0.44$, then the results
in \cite{JSW1993} and \cite{DN1993} can be used to obtain that $\cMq$
is solid in the sense of Ozawa.

Based on the results in Section \ref{sec:main-result}, we show here
that $\cMq$ is weakly exact. For more details on weak exactness,
we refer the reader to Chapter 14 of \cite{BO2008}.
\begin{thm}
\label{thm:weakly-exact}For every $q$ in the interval $(-1,1)$,
the $q$-Gaussian von Neumann algebra $\cMq$ is weakly exact. \end{thm}
\begin{proof}
It is known that a von Neumann algebra is weakly exact if it contains
a weakly dense $\mathrm{C}^{*}$-algebra which is exact (see e.g.
Theorem 14.1.2 of \cite{BO2008}). Consider the $\mathrm{C}^{*}$-algebra
$\cAq$ generated by $\{\Lq_{i}+(\Lq_{i})^{*}\mid1\leq i\leq d\}$.
It is clear that $\cAq$ is weakly dense in $\cMq$, while on the
other hand, we have $\cAq\subseteq\cCq$. Therefore, the exactness
of $\cAq$ follows from Corollary \ref{cor:Cq-exact}, combined with
the fact that exactness is inherited by subalgebras.\end{proof}

\end{document}